\documentclass{amsart}
\usepackage{amssymb}
\usepackage{enumerate}
\usepackage{hyperref}
\usepackage[thinlines]{easytable}

\newtheorem{theorem}{Theorem}[section]
\newtheorem{lemma}[theorem]{Lemma}
\newtheorem{corollary}[theorem]{Corollary}

\theoremstyle{definition}
\newtheorem{remark}[theorem]{Remark}

\newcommand{\Aut}{\mathop{\mathrm{Aut}}}
\newcommand{\Alt}{\mathop{\mathrm{Alt}}}
\newcommand{\Sym}{\mathop{\mathrm{Sym}}}
\newcommand{\soc}{\mathop{\mathrm{soc}}}
\newcommand{\Suz}{\mathop{\mathrm{Suz}}}

\newcommand{\SL}{\mathop{\mathrm{SL}}}
\newcommand{\SU}{\mathop{\mathrm{SU}}}

\newcommand{\GL}{\mathop{\mathrm{GL}}}
\newcommand{\Sp}{\mathop{\mathrm{Sp}}}
\newcommand{\AGL}{\mathop{\mathrm{AGL}}}
\newcommand{\PSL}{\mathop{\mathrm{PSL}}}
\newcommand{\PGL}{\mathop{\mathrm{PGL}}}
\newcommand{\PSU}{\mathop{\mathrm{PSU}}}

\newcommand{\PSp}{\mathop{\mathrm{PSp}}}
\newcommand{\PGSp}{\mathop{\mathrm{PGSp}}}
\newcommand{\GSp}{\mathop{\mathrm{GSp}}}
\newcommand{\GammaL}{\mathop{\mathrm{\Gamma L}}}
\newcommand{\AGammaL}{\mathop{\mathrm{A\Gamma L}}}

\newcommand{\PGammaL}{\mathop{\mathrm{P\Gamma L}}}
\newcommand{\PSigmaL}{\mathop{\mathrm{P\Sigma L}}}
\newcommand{\POmega}{\mathop{\mathrm{P\Omega}}}

\newcommand{\GF}{\mathop{\mathrm{GF}}}
\newcommand{\vGa}{\vec{\Gamma} }
\newcommand*{\nonsplit}{\,{}^{\displaystyle .}}
\newcommand{\End}{\mathop{\mathrm{End}}}

\newcommand{\magma}{{\sc Magma}}
\newcommand{\GAP}{{\textsf GAP}}

\def\D{{\rm D}}
\def\C{{\rm C}}
\def\Z{{\rm Z}}
\def\V{{\rm V}}
\def\J{{\rm J}}
\def\Th{{\rm Th}}
\def\M{{\rm M}}
\def\N{{\rm N}}
\newcommand{\ZZ}{\mathbb{Z}}

\newcommand{\thmp}{{\rm Th}}

\renewcommand{\geq}{\geqslant}
\renewcommand{\leq}{\leqslant}

\begin{document}
\title[Primitive groups with a suborbit of length $5$ ]{Primitive permutation groups with a suborbit of length $5$ and vertex-primitive graphs of valency $5$}
\author[J. B. Fawcett, M. Giudici, C. H. Li, C. E. Praeger, G. Royle, G. Verret]{Joanna B. Fawcett, Michael Giudici, Cai Heng Li, Cheryl E. Praeger, Gordon Royle, Gabriel Verret}

\address{Joanna B. Fawcett, Michael Giudici, Cai Heng Li, Cheryl E. Praeger, Gordon Royle
\newline\indent and Gabriel Verret$^*$,
\newline\indent Centre for the Mathematics of Symmetry and Computation, 
\newline\indent The University of Western Australia, 
\newline\indent 35 Stirling Highway, Crawley, WA 6009, Australia.} 

\address{$^*$Also affiliated with FAMNIT, University of Primorska, 
\newline\indent Glagolja\v{s}ka 8, SI-6000 Koper, Slovenia.
\newline\indent Current address: Department of Mathematics, The University of Auckland,
\newline\indent Private Bag 92019, Auckland 1142, New Zealand.}

\email{Joanna.Fawcett@uwa.edu.au}
\email{Michael.Giudici@uwa.edu.au}
\email{Cai.Heng.Li@uwa.edu.au}
\email{Cheryl.Praeger@uwa.edu.au}
\email{Gordon.Royle@uwa.edu.au}
\email{g.verret@auckland.ac.nz}

\begin{abstract}
We classify finite primitive permutation groups having a suborbit of length $5$. As a corollary, we obtain a classification of finite vertex-primitive graphs of valency $5$. In the process, we also classify finite almost simple groups that have a maximal subgroup isomorphic to $\Alt(5)$ or $\Sym(5)$. 
\end{abstract}

\maketitle

\section{Introduction}

All graphs and groups considered in this paper are finite. Let $G$ be a transitive permutation group on a set $\Omega$ and let $\omega\in\Omega$. An orbit of the point-stabiliser $G_\omega$ is called a \emph{suborbit} of $G$ and it is \emph{non-trivial} unless it is $\{\omega\}$.
It is an easy exercise to show that if a primitive permutation group has a non-trivial suborbit of length one, then it must be regular of prime order, while if it has a suborbit of length two, it must be dihedral of degree an odd prime. 

Primitive groups with a suborbit of length three have a more complicated structure. Classifying them is non-trivial and was accomplished by Wong~\cite{Wong} using the work of Sims~\cite{Sims}. The classification of primitive groups with a suborbit of length four is even more difficult. After some partial results by Sims~\cite{Sims} and Quirin~\cite{Quirin}, this was finally completed by Wang~\cite{Wang} using the classification of finite simple groups.

Wang then turned his attention to the case of primitive groups with a suborbit of length $5$. He proved some strong partial results~\cite{WangSoluble5,WangA5} but was unable to complete this project. This classification is the main result of our paper.

\begin{theorem}\label{theo:mainOrbital}
A primitive permutation group $G$ has a suborbit of length $5$ if and only if $(G,G_v)$ appears in Table~\ref{table:sporadic} or~\ref{table:infinite}.
\end{theorem}

Note that  each row in Table~\ref{table:sporadic} corresponds to a unique primitive permutation group $G$ with a suborbit of length $5$ whereas, in Table~\ref{table:infinite}, there exists one group for each value of the parameter $p$. (Throughout this paper, $p$ always denotes a prime, while $\D_n$ denotes a dihedral group of order $2n$.)

\begin{table}[hhhhhh]
\begin{center}
\begin{TAB}(@){cccc}{|c|ccccccccccccc|}
     
& $G$ & $G_v$  & $|G:G_v|$ \\   
(1)&    $\Alt(5)$ & $\D_{5}$  & $6$ \\ 
(2)&        $\Sym(5)$ & $\AGL(1,5)$  & $6$ \\                    
(3)&   $\PGL(2,9)$ & $\D_{10}$ & $36$\\
(4)&$\M_{10}$ & $\AGL(1,5)$ & $36$\\
(5)&$\PGammaL(2,9)$ & $\AGL(1,5)\times\ZZ_2$ & $36$\\
(6)& $\PGL(2,11)$ & $\D_{10}$ & $66$\\  
(7)&    $\Alt(9)$ & $(\Alt(4)\times\Alt(5))\rtimes\ZZ_2$ & $126$\\ 
(8)&  $\Sym(9)$ & $\Sym(4)\times\Sym(5)$ & $126$\\ 
(9)&    $\PSL(2,19)$ & $\D_{10}$ & $171$\\  
(10)& $\Suz(8)$ & $\AGL(1,5)$ & $1\,456$ \\ 
 (11)&  $\J_3$ & $\AGL(2,4)$  & $17\,442$\\ 
(12)&  $\J_3\rtimes\ZZ_2$ & $\AGammaL(2,4)$  & $17\,442$\\ 
(13)&  $\Th$ & $\Sym(5)$  & $756\,216\,199\,065\,600$ \\      
\end{TAB}
\caption{\small{Primitive groups with a suborbit of length $5$: sporadic examples.}}\label{table:sporadic}
\end{center}
\end{table}

\begin{table}[hhhhhh]
\begin{center}
\begin{TAB}(@){ccccc}{|c|cccccccccccccc|}
    
  & $G$ & $G_v$  & $|G:G_v|$ & Conditions  \\   
  (1)& $\ZZ_p\rtimes\ZZ_5$ & $\ZZ_5$ & $p$& $p\equiv 1 \pmod 5$\\ 
(2)& $\ZZ_p^2\rtimes\ZZ_5$ & $\ZZ_5$ & $p^2$& $p\equiv -1 \pmod 5$\\ 
(3)& $\ZZ_p^4\rtimes\ZZ_5$ & $\ZZ_5$ & $p^4$& $p\equiv \pm2 \pmod 5$\\ 
(4)& $\ZZ_p^2\rtimes\D_5$ & $\D_5$ & $p^2$& $p\equiv \pm1 \pmod 5$\\ 
(5)& $\ZZ_p^4\rtimes\D_5$ & $\D_5$ & $p^4$& $p\equiv \pm2 \pmod 5$\\ 
(6)& $\ZZ_p^4\rtimes\AGL(1,5)$ & $\AGL(1,5)$ & $p^4$& $p\neq 5$\\ 
(7)& $\ZZ_p^4\rtimes\Alt(5)$ & $\Alt(5)$  & $p^4$&$p\neq 5$\\             
(8)& $\ZZ_p^4\rtimes\Sym(5)$ & $\Sym(5)$  & $p^4$&$p\neq 5$\\
(9)&  $\PSL(2,p)$ & $\Alt(5)$ & $\frac{p^3-p}{120}$& $p\equiv \pm 1,\pm 9 \pmod{40}$\\ 
(10)& $\PSL(2,p^2)$ & $\Alt(5)$ & $\frac{p^6-p^2}{120}$& $p\equiv\pm 3\pmod {10}$ \\ 
(11)&  $\PSigmaL(2,p^2)$ & $\Sym(5)$ & $\frac{p^6-p^2}{120}$& $p\equiv\pm 3\pmod {10}$ \\ 
(12)& $\PSp(6,p)$ & $\Sym(5)$  & $\frac{p^9(p^6-1)(p^4-1)(p^2-1)}{240}$&$p\equiv\pm 1 \pmod{8}$ \\
(13)& $\PSp(6,p)$ & $\Alt(5)$  & $\frac{p^9(p^6-1)(p^4-1)(p^2-1)}{120}$ & $p\equiv \pm 3,\pm 13 \pmod{40}$ \\
(14)& $\PGSp(6,p)$ & $\Sym(5)$  & $\frac{p^9(p^6-1)(p^4-1)(p^2-1)}{120}$&$p\equiv  \pm 3 \pmod{8}$, $p\geq 11$ \\
\end{TAB}
\caption{\small{Primitive groups with a suborbit of length $5$: infinite families.}}\label{table:infinite}
\end{center}
\end{table}

The classification of primitive groups with suborbits of length three or four was used by Li, Lu and Maru\v{s}i\v{c}~\cite{LiLuMar} to obtain a classification of arc-transitive vertex-primitive graphs of valency three or four. Similarly, as an application of Theorem~\ref{theo:mainOrbital}, we prove the following:

\begin{theorem}\label{theo:mainGraph}
A $5$-valent graph $\Gamma$ is vertex-primitive if and only if $(\Aut(\Gamma),\Aut(\Gamma)_v)$ appears in Table~\ref{table:main}.
\end{theorem}

Note that a few well-known graphs appear in Table~\ref{table:main}: the Clebsch graph in row (1), the Sylvester graph in row (2), the Odd graph $\mathrm{O}_5$ in row (4) and, when $p=3$, the complete graph on $6$ vertices in row (9) (recall that $\PSigmaL(2,9)\cong\Sym(6)$).

\begin{table}[hhhhh]
\begin{center}
\begin{TAB}(@){ccccc}{|c|ccccccccccc|}
& $\Aut(\Gamma)$ & $\Aut(\Gamma)_v$  & $|\V(\Gamma)|$ &Conditions\\                                       
 (1)&   $\ZZ_2^4\rtimes\Sym(5)$ & $\Sym(5)$  & 16&\\ 
(2)&    $\PGammaL(2,9)$ & $\AGL(1,5)\times\ZZ_2$ & 36&\\
(3)& $\PGL(2,11)$ & $\D_{10}$ & 66&\\  
(4)&   $\Sym(9)$ & $\Sym(4)\times\Sym(5)$ & 126& \\ 
(5)& $\Suz(8)$ & $\AGL(1,5)$ & $1\,456$& \\ 
(6)&  $\J_3\rtimes2$ & $\AGammaL(2,4)$  & $17\,442$&\\ 
(7)&  $\Th$ & $\Sym(5)$  & $756\,216\,199\,065\,600$& \\ 
(8)&  $\PSL(2,p)$ & $\Alt(5)$ & $\frac{p^3-p}{120}$& $p\equiv \pm 1,\pm 9 \pmod{40}$\\
(9)&  $\PSigmaL(2,p^2)$ & $\Sym(5)$ & $\frac{p^6-p^2}{120}$& $p\equiv\pm 3\pmod {10}$ \\ 
(10)&  $\PSp(6,p)$ & $\Sym(5)$ & $\frac{p^9(p^6-1)(p^4-1)(p^2-1)}{240}$&$p\equiv\pm 1 \pmod{8}$ \\
(11)&  $\PGSp(6,p)$ & $\Sym(5)$ & $\frac{p^9(p^6-1)(p^4-1)(p^2-1)}{120}$&$p\equiv  \pm 3 \pmod{8}$, $p\geq 11$ \\ 
\end{TAB}
\caption{\small{Vertex-primitive graphs of valency $5$.}}\label{table:main}
\end{center}
\end{table}

In the process of proving Theorem~\ref{theo:mainOrbital}, we are led to classify almost simple groups admitting a maximal subgroup isomorphic  to $\Alt(5)$ or $\Sym(5)$.

\begin{theorem}\label{theo:max}
An almost simple group $G$ has a maximal subgroup $M$ isomorphic to $\Alt(5)$ or $\Sym(5)$ if and only if $G$ appears in Table~\ref{table:maxA5} or~\ref{table:maxS5}, respectively. Moreover, the third column in these tables gives the number $c$ of conjugacy classes of such subgroups in $G$, while the fourth column gives the structure of $\N_G(H)/H$, where $H$ is a subgroup of index $5$ in $M$.
\end{theorem}

\begin{table}[hhhhh]
\begin{center}
\begin{TAB}(@){ccccc}{|c|cccccccccc|}
 & $G$  &$c$& $\N_G(H)/H$&Conditions\\
 (1) &$\Sym(5)$ &$1$& $\ZZ_2$&  \\                                   
(2) & $\J_2$  &$1$&$1$ & \\           
 (3) &$\PSL(2,p)$ &$2$& $1$& $p\equiv\pm 11, \pm 19 \pmod{40}$\\ 
 (4) &$\PSL(2,p)$ &$2$& $\ZZ_2$& $p\equiv\pm 1, \pm 9 \pmod{40}$\\ 
 (5) &$\PSL(2,p^2)$ &$2$& $\ZZ_2$& $p\equiv\pm 3 \pmod{10}$\\ 
 (6) &$\PSL(2,2^{2r})$ &$1$&$1$& $r$ prime\\ 
 (7) &$\PSL(2,5^r)$ &$1$&$1$& $r$ odd prime\\  
(8)&$\PSp(6,3)$ &$1$ & $\ZZ_3$ &  \\
(9)&  $\PSp(6,p)$ &$1$ &$\ZZ_{p-1}$& $p\equiv 13,37,43,67 \pmod{120}$\\ 
(10)&$\PSp(6,p)$ &$1$ &$\ZZ_{p+1}$& $p\equiv 53,77,83,107 \pmod{120}$\\  
\end{TAB}
\caption{\small{Almost simple groups with maximal $\Alt(5)$.}}\label{table:maxA5}
\end{center}
\end{table}

\begin{table}[hhhhh]
\begin{center}
\begin{TAB}(@){ccccc}{|c|cccccccccccccc|}
     
 & $G$ &$c$ & $\N_G(H)/H$&Conditions\\                    
(1) &    $\Alt(7)$ &$1$& $1$&  \\ 
(2) &  $\M_{11}$ &$1$& $1$&  \\     
(3) &    $\M_{12}\rtimes\ZZ_2$ &$1$&$1$ &  \\ 
(4) &  $\J_2\rtimes\ZZ_2$ &$1$& $1$&  \\  
(5) &  $\Th$  &$1$& $\ZZ_2$&\\ 
(6) &  $\PSL(2,5^2)$&$2$&$1$ &  \\  
(7) &  $\PSigmaL(2,p^2)$ &$2$& $\ZZ_2$& $p\equiv\pm 3 \pmod{10}$\\ 
(8) &  $\PSL(2,2^{2r})\rtimes\ZZ_2$&$1$&$1$ & $r$ odd prime\\ 
(9) &  $\PGL(2,5^r)$&$1$&$1$ & $r$ odd prime\\ 
(10) &  $\PSL(3,4)\rtimes\langle\sigma\rangle$&$1$&$1$ & $\sigma$ a graph-field automorphism \\                  
(11) &  $\PSL(3,5)$&$1$& $1$& \\ 
(12) & $\PSp(6,p)$ &$2$&$\ZZ_2$& $p\equiv\pm 1 \pmod{8}$ \\ 
 (13) & $\PGSp(6,3)$ &$1$& $1$&  \\ 
 (14) & $\PGSp(6,p)$ &$1$& $\ZZ_2$& $p\equiv\pm 3 \pmod{8}$, $p\geq 11$ \\ 
\end{TAB}
\caption{\small{Almost simple groups with maximal $\Sym(5)$.}}\label{table:maxS5}
\end{center}
\end{table}


As a consequence of our results, we are also able to prove the following  two corollaries. (A graph is called \emph{half-arc-transitive} if its automorphism group acts transitively on its vertex-set and on its edge-set, but not on its arc-set.) 

\begin{corollary}\label{cor:halfarc}
There is no half-arc-transitive vertex-primitive graph of valency $10$.
\end{corollary}

\begin{corollary}\label{cor:halfarc2}
There are infinitely many half-arc-transitive vertex-primitive graphs of valency $12$.
\end{corollary}

It is easy to see that a half-arc-transitive graph must have even valency. In \cite{LiLuMar}, it was proved that there is no vertex-primitive example of valency at most $8$. The two results above thus imply that $12$ is the smallest valency for a half-arc-transitive vertex-primitive graph, solving~\cite[Problem~1.3]{LiLuMar}.

After some preliminaries in Section~\ref{sec:prelim}, we prove Theorem~\ref{theo:mainOrbital} in Section~\ref{sec:firstmain} and Theorem~\ref{theo:mainGraph} in Section~\ref{sec:secondmain}. These proofs are conditional on the proof of Theorem~\ref{theo:max} which, being slightly more technical, is delayed until Section~\ref{sec:thirdmain}. We then prove Corollary~\ref{cor:halfarc} in Section~\ref{sec:HAT} and Corollary~\ref{cor:halfarc2} in Section~\ref{sec:HAT2}. Before moving on to these proofs, we correct a few mistakes in the literature on this subject that have, as far as we know, gone undetected until now.

\subsection{Corrections}
\begin{itemize}

\item In~\cite[Theorem~1.2]{Li}, it is mistakenly claimed that there exists an infinite family of $5$-valent vertex-primitive $4$-arc-transitive graphs. In fact, as can be inferred from Table~\ref{table:main}, there is a unique such graph and it has order $17442$.

\item In~\cite[Table~2]{LiLuMar}, in the first row, one should have $p\equiv 1\pmod{4}$. In the third row, the condition ``$p \equiv \pm 1 \pmod{8}$, $p\neq 7$'' should be replaced by ``$p \equiv \pm 1 \pmod{24}$''. In the fifth row, one should have $p\neq 3$. Finally, in the last row, $\Aut(\Gamma)$ should be $\PSL(3,7).2$ and the vertex-stabiliser should be $\Sym(4)\times\Sym(3)$. (See also next item.)

\item In~\cite[Table~3]{LiLuMar}, the case when $G=\PSL(3,7).\langle\sigma\rangle$ where $\sigma$ is a graph automorphism is missing. (This can be traced back to a typographical error in~\cite[Theorem~1.4(6)]{Wang} where it should read $\PSL(3,7).2$ instead of $\PSL(3,7).3$. Note that the correct version appears in Theorem~1.3(2) of the same paper.) 

\item In the main theorem of~\cite{WangA5}, the case $\soc(G)=\J_3$ is missing. (Indeed, this example was already known to Weiss~\cite{Weiss}.)
\end{itemize}

\begin{remark}
As the above list reminds us, it is easy for a mistake to slip in with these kinds of results. This is partly because of the nature of the proofs. To increase our confidence in the correctness of our results, we have checked them against databases of known examples whenever possible. More specifically, we have checked Tables~\ref{table:sporadic}--\ref{table:main} against the database of primitive groups of degree less than $4\,096$~\cite{Colva}, and  Tables~\ref{table:maxA5} and~\ref{table:maxS5} against the database of almost simple groups of order at most $16\,000\,000$ implemented in~\magma~\cite{magma}.
\end{remark}

\section{Preliminaries}\label{sec:prelim}

In this section, as well as in Sections~\ref{sec:HAT} and~\ref{sec:HAT2}, we will need the notion of a digraph. Since this terminology has many usages, we formalise ours here. A \emph{digraph} $\Gamma$ is a pair $(V,A)$ where $V$ is a set and $A$ is a binary relation on $V$. The set $V$ is called the \emph{vertex-set} of $\Gamma$ and its elements are the \emph{vertices}, while $A$ is the \emph{arc-set} and its elements \emph{arcs}. If $A$ is a symmetric relation, then $\Gamma$ is called a \emph{graph}. 

If $(u,v)\in A$, then $v$ is an \emph{out-neighbour} of $u$ and $u$ is an \emph{in-neighbour} of $v$. The number of out-neighbours of $v$ is its \emph{out-valency}. If this does not depend on the choice of $v$, then it is the \emph{out-valency} of $\Gamma$. An \emph{automorphism} of $\Gamma$ is a permutation of $V$ that preserves $A$. We say that $\Gamma$ is \emph{$G$-arc-transitive} if $G$ is a group of automorphisms of $\Gamma$ that acts transitively on  $A$.

The following easy lemma will prove useful. Here, for a (not necessarily normal) subgroup $H$ of a group $G$, $G/H$ denotes the set of right $H$-cosets in $G$.

\begin{lemma}\label{lemma:selfnormalising}
Let $d\geq 2$ and let $G$ be a non-regular primitive permutation group such that $G_v$ has a unique conjugacy class of subgroups of index $d$, and these subgroups are maximal and self-normalising  in $G_v$. Let $H$ be a representative of this conjugacy class, and let $N=\N_G(H)$ be the normaliser of $H$ in $G$. The following hold.
\begin{enumerate}
\item There is a one-to-one correspondence between $G$-arc-transitive digraphs of out-valency $d$ and elements $Hg$ of $(N/H)\setminus \{H\}$.
\item Such a digraph is a graph if and only if $Hg$ has order $2$ in $N/H$.
\item $G_v$ has an orbit of length $d$ if and only if $(N/H)\setminus \{ H\}\neq \emptyset$ or, equivalently, $N> H$.
\end{enumerate}
\end{lemma}
\begin{proof}
We prove the three claims in order.
\begin{enumerate}
\item \label{one} Let $\Gamma$ be a $G$-arc-transitive digraph of out-valency $d$.  Since $\Gamma$ is $G$-arc-transitive, $\Gamma$ also has in-valency $d$. Also $G_v$ is transitive on the $d$ in-neighbours of $v$ and hence the stabilisers of these in-neighbours form a conjugacy class of subgroups of $G_v$ of index $d$. As there is a unique such conjugacy class and $H$ is contained in it, we have $H=G_{uv}$ for some in-neighbour $u$ of $v$. 
Since $H$ is a self-normalising proper subgroup of $G_v$, it follows that $u$ is the unique in-neighbour of $v$ fixed by $H$ (this is easily proved, and, for example, is set as an exercise in \cite[Exercise 1.6.3]{DM}).
The same argument on out-neighbours shows that $H=G_{vw}$ for a unique out-neighbour $w$ of $v$. Since $\Gamma$ is $G$-arc-transitive, there exists a unique coset $Hg\in G/H$ such that $(u,v)^g=(v,w)$.  Note that $H^g=G_{uv}^g=G_{vw}=H$, and hence $Hg\in N/H$. Also $u\neq v$, and hence $Hg\ne H$. Thus $\varphi:\Gamma\rightarrow Hg$ is a well-defined map from $G$-arc-transitive digraphs of out-valency $d$ to  $(N/H)\setminus \{ H\}$. We show that $\varphi$ is a bijection.

To show that $\varphi$ is onto, let $Hg\in(N/H)\setminus \{ H\}$, and let $w=v^g$. Since $g\notin G_v$, $w\neq v$. Let $\Gamma$ be the $G$-arc-transitive digraph with arc-set $(v,w)^G$. We have $H=H^g\leq (G_v)^g=G_w$ and thus $H\leq G_{vw}$. Since $G$ is primitive but not regular, we have  $G_{vw}< G_v$. As $H$ is maximal in $G_{v}$, it follows that  $H=G_{vw}$, and hence $\Gamma$ has out-valency $d$. Finally, as in the previous paragraph, $G_{vw}=H = G_{uv}$ for some unique in-neighbour $u$ of $v$, and we have $(u,v)^g=(v,w)$ so $\varphi(\Gamma)=Hg$.

To show that $\varphi$ is one-to-one, suppose that $\Gamma =(V,A)$ and $\Delta =(V, B)$ are $G$-arc-transitive digraphs of out-valency $d$ with images $\varphi(\Gamma)=Hg$ and
 $\varphi(\Delta)=Hk$ such that $Hg=Hk$. By the first paragraph of the proof, $H = G_{uv}$ where $A=(u,v)^G$, $u^g=v$ and $g\in N\setminus G_v$, and also $H=G_{xv}$, where   $B=(x,v)^G$, $x^k=v$ and $k\in N\setminus G_v$. Since $Hg=Hk$, we have $k=hg$ for some $h\in H$, and so $u^g = v = x^k = x^{hg} = x^g$. Hence $u=x$ which implies that $A=B$ and $\Gamma = \Delta$.  

\item \label{two} Suppose that $\Gamma$ is a $G$-arc-transitive graph of valency $d$. Adopting the notation from the first paragraph of the proof of~(\ref{one}), we can choose $u$ and $w$ such that $u=w$. Hence $(u,v)^g=(v,u)$ and $g^2\in G_{uv}=H$. In other words, $Hg$ has order $2$.

Conversely, if $Hg$ is an element of order $2$ in $(N/H)\setminus \{ H\}$ then, adopting the notation from the second paragraph of the proof of~(\ref{one}), we have that $g^2\in H=G_{vw}$. Hence $w^g=v^{g^2}=v$ and $(w,v)=(v,w)^g$, so $\Gamma$ is a graph.

\item \label{three} Suppose that $G_v$ has an orbit of length $d$. Let $w$ be an element of that orbit, and let $\Gamma$ be the digraph with arc-set $(v,w)^G$. Clearly, $\Gamma$ is $G$-arc-transitive and has out-valency $d$. Hence, by~(\ref{one}), $(N/H)\setminus \{ H\}\neq \emptyset$.

Conversely, if $(N/H)\setminus \{ H\}\neq \emptyset$, then, by~(\ref{one}), there exists a $G$-arc-transitive digraph of out-valency $d$ and thus $G_v$ has an orbit of length~$d$. \qedhere
\end{enumerate}
\end{proof}

\subsection{Brauer characters}

We will often use  the Brauer character tables of $\Sym(n)$, $\Alt(n)$ and their double covers for $n=4$ or $5$. For an algebraically closed field $F$ of characteristic $p$ and a group $G$, the Brauer character of a finite-dimensional $F$-representation  $\varphi$ of $G$  is a  function that maps each $p$-regular element $g$ of $G$ to the sum of lifts to $\mathbb{C}$ of  the eigenvalues of $\varphi(g)$ (see \cite{Isaacs} for definitions). The Brauer character of $\varphi$ is the sum of the Brauer characters of the irreducible constituents of $\varphi$, and two irreducible representations are equivalent precisely when their Brauer characters are equal \cite[Theorem 15.5]{Isaacs}, which occurs exactly when their characters are equal \cite[Corollary 9.22]{Isaacs}. The Brauer character table  describes the irreducible Brauer characters. If $p\nmid |G|
$, then the Brauer table is the same as the complex character table \cite[Theorem 15.13]{Isaacs} and, for the groups above, can be accessed in    \GAP~\cite{gap} or \magma~\cite{magma}, as well as the  Atlas~\cite{atlas} when $n=5$. Otherwise, the Brauer table can be accessed in \GAP, or the Brauer Atlas~\cite{modatlas} when $n=5$. See  \cite{modatlas} for details on how to read the tables, and  \cite[Section 4.2]{BHRD} for information about irrationalities.

The following theory will be used in conjunction with Brauer character tables. Let $G$ be a group and $F$ a field.  An irreducible $FG$-module $V$ is \textit{absolutely irreducible} if  the extension of scalars $V\otimes E$ is irreducible for every field extension $E$ of $F$. Note that $V$ is absolutely irreducible if and only if  $\End_{FG}(V)=F$ \cite[Lemma VII.2.2]{HB}, where $\End_{FG}(V)$ denotes the ring of $FG$-endomorphisms of $V$. The field $F$ is a \textit{splitting field} for $G$ if every irreducible $FG$-module is absolutely irreducible. For a character $\chi$ of an $FG$-module $V$ and a subfield $K$ of $F$, let $K(\chi)$ denote the subfield of $F$ generated by $K$ and the image of $\chi$.

Now suppose that $G$ is one of the groups above.  By the Brauer character tables of these groups and \cite[Theorem VII.2.6]{HB}, $F=\GF(q^2)$ is a splitting field for $G$ for any  prime power $q$, so every irreducible representation of $G$ over the algebraic closure of $F$ can be realised over $F$. Let $K=\GF(q)$. If $V$ is an irreducible $FG$-module with character $\chi$, then either  $K(\chi)=K$ and $V=U\otimes F$ for some absolutely irreducible $KG$-module $U$ \cite[Theorem VII.1.17]{HB}, or $K(\chi)=F$ and $V$ is an irreducible $KG$-module of dimension $2\dim_F(V)$ \cite[Theorem  VII.1.16]{HB}. Moreover, every irreducible $KG$-module arises in this way. Indeed, suppose that $V$ is an irreducible $KG$-module that is not absolutely irreducible,  and let $k:=\End_{KG}(V)$. Then $k$ is a finite field by Wedderburn's Theorem, and $V$ is an absolutely irreducible $kG$-module where  $k$-scalar multiplication is defined to be evaluation. Let $\chi$ be the character of $V$ as a $kG$-module. Then $k=K(\chi)$ by \cite[Theorem  VII.1.16]{HB}, and $K(\chi)\subseteq F$ by \cite[Theorem VII.2.6]{HB} (or the Brauer tables). Hence $k=F$ and $V$ is an irreducible $FG$-module, as desired. Further,  $V\otimes F$ (with $V$ as a $KG$-module) is a direct sum of two non-isomorphic irreducible $FG$-modules with the same dimension, one of which is $V$ as an $FG$-module  \cite[Lemma VII.1.15 and Theorem VII.1.16]{HB}.

\section{Proof of Theorem~\ref{theo:mainOrbital}}\label{sec:firstmain}
For the rest of this section, let $G$ be a primitive group, let $G_v$ be one of its point-stabilisers, let $\Delta$ be an orbit of $G_v$ of length $5$ and let $G_v^{\Delta}$ be the permutation group induced by the action of $G_v$ on $\Delta$.

We first report the  following result of Wang~\cite{WangSoluble5}. (Note that the case corresponding to Table~\ref{table:sporadic} (2) was mistakenly omitted in~\cite{WangSoluble5}.)

\begin{theorem}\label{theo:GvSoluble}
$G_v^{\Delta}$ is soluble if and only if $(G,G_v)$ appears in Table~\ref{table:sporadic} (1--6,9,10) or Table~\ref{table:infinite} (1--6).
\end{theorem}

It thus remains to consider the case when  $G_v^{\Delta}$ is not soluble. Since it is a transitive permutation group of degree $5$, it must be isomorphic to either $\Alt(5)$ or $\Sym(5)$. We first consider the case when $G_v$ does not act faithfully on $\Delta$.

\begin{theorem}\label{theo:GvInsolubleUnfaithful}
$G_v^{\Delta}\in\{\Alt(5),\Sym(5)\}$ and $G_v$ does not act faithfully on $\Delta$ if and only if $(G,G_v)$ appears in Table~\ref{table:sporadic} (7,8,11,12).
\end{theorem}
\begin{proof}
This is essentially a result of Wang~\cite{WangA5}, except that the author left open the case when $G$ is isomorphic to one of the Monster or Baby Monster sporadic groups and $G_v$ is a maximal $2$-local subgroup of $G$. By \cite[Theorem 5.2]{Knapp76} or \cite{WangA5}, the order of $G_v$ divides $2^{14}\cdot3^2\cdot5$.
This is impossible for the Monster by \cite{BrayWilson}, and for the Baby Monster by \cite{Wilsonmax}. Moreover, while running some computations, we noticed that Wang  mistakenly excluded the cases corresponding to Table~\ref{table:sporadic} (11,12).
\end{proof}

By Theorems~\ref{theo:GvSoluble} and~\ref{theo:GvInsolubleUnfaithful}, it suffices to consider the case when $G_v\cong G_v^{\Delta}\in\{\Alt(5),\Sym(5)\}$. Since $\Alt(5)$ and $\Sym(5)$ are $2$-transitive, it follows by~\cite[Theorem A]{2AT} that either $G$ is almost simple, or it has a  unique minimal normal subgroup which is regular. We deal with the latter case in the next two results. (Recall that a primitive group is \emph{affine} if it has an elementary abelian regular normal subgroup.)

\begin{lemma}\label{theo:TW}
If $G_v\in\{\Alt(5),\Sym(5)\}$ and $G$ has a unique minimal normal subgroup which is regular, then $G$ is affine.
\end{lemma}
\begin{proof}
Let $N$ be the unique minimal normal subgroup of $G$. If $N$ is abelian, then $G$ is affine. We thus assume that $N$ is non-abelian and hence $N=T^m$ for some non-abelian simple group $T$. Write $N=T_1\times\cdots \times T_m$ and let $X=\N_{G_v}(T_1)$.

By \cite[Theorem 4.7B]{DM}, $m\geq 6$,  the action by conjugation of $G_v$ on $\{T_1,\ldots,T_m\}$ is faithful and transitive, and $X$ has a composition factor isomorphic to $T$. The only non-abelian composition factor of $G_v$ is $\Alt(5)$ and thus $m=|G_v:X|\leq 2$, which is a contradiction. 
\end{proof}

\begin{lemma}\label{theo:affine}
$G$ is affine and $G_v\in\{\Alt(5),\Sym(5)\}$  if and only if $(G,G_v)$ appears in Table~\ref{table:infinite} (7,8).
\end{lemma}
\begin{proof}
We assume that $G$ is affine and $G_v\in\{\Alt(5),\Sym(5)\}$. By definition, $G$ has an elementary abelian regular normal subgroup $V$, with $V\cong\ZZ_p^d$. Note that $G=V\rtimes G_v$. We view $V$ as a faithful irreducible $\GF(p)G_v$-module. 

Let $H$ be a subgroup of index $5$ in $G_v$ and let $\C_V(H)$ be the centraliser of $H$ in $V$. Since $H$ is self-normalising in $G_v$, it follows that $\N_G(H)=\C_V(H)\rtimes H$. Lemma~\ref{lemma:selfnormalising}(3) implies that $\C_V(H)\neq 0$, so the trivial $\GF(p)H$-module is a submodule of $V$.

Suppose that $p>5$ and $G_v=\Sym(5)$. In this case, $V$ is isomorphic to a Specht module $S^\mu$ for some partition $\mu$ of $5$. Since the trivial $\GF(p)H$-module is a submodule of $V$, \cite[Theorem 9.3]{James} implies that  we can remove an element from one of the parts of $\mu$ and obtain the partition $(4)$. If $\mu=(5)$, then $V$ is the trivial module, a contradiction. Hence $\mu=(4,1)$, in which case $d=4$ and $(G,G_v)$ appears in Table~\ref{table:infinite}~(8), and conversely, the pair $(G,G_v)$ has the required properties.

Suppose that $p>5$ and $G_v=\Alt(5)$. In particular, $H=\Alt(4)$. Using the Brauer character tables of $\Alt(4)$ and $\Alt(5)$, we determine that  $d=4$. Hence  $(G,G_v)$ appears in Table~\ref{table:infinite}~(7), and conversely, the pair $(G,G_v)$ has the required properties.

Finally, suppose that $p\leq 5$. Using \magma, we determine that $p\leq 3$ and $V$ is the deleted permutation module. Hence $d=4$ and $(G,G_v)$ appears in Table~\ref{table:infinite}~(7,8), and conversely, the pairs $(G,G_v)$ are examples.
\end{proof}

By Lemmas~\ref{theo:TW} and~\ref{theo:affine} and the remark preceding them, we may now assume that $G$ is an almost simple group. In particular, by Theorem~\ref{theo:max}, the possible groups $G$ appear in Table~\ref{table:maxA5} or~\ref{table:maxS5}. In view of Lemma~\ref{lemma:selfnormalising}~(\ref{three}), Theorem~\ref{theo:mainOrbital} now follows by going through these tables and ignoring the rows with $\N_G(H)/H=1$. (Row (1) of Table~\ref{table:maxA5} must also be ignored as $M$ is not core-free in $G$ in this case.)

\section{Proof of Theorem~\ref{theo:mainGraph}}\label{sec:secondmain}

Let $\Gamma$ be a $5$-valent vertex-primitive graph and let $G=\Aut(\Gamma)$. We first show that $\Gamma$ is $G$-arc-transitive. Suppose, on the contrary, that $\Gamma$ is not $G$-arc-transitive and thus $G_v^{\Gamma(v)}$ is intransitive. If $G_v^{\Gamma(v)}$ has a fixed point then, since $G$ is primitive, it is regular and cyclic of prime order at least $7$. However, a non-trivial regular abelian group $G$ of odd order cannot be the full automorphism group of a graph since the permutation sending each element to its inverse is a nontrivial automorphism with a fixed point.  Thus  $G_v^{\Gamma(v)}$ has two orbits, one of length $2$ and one of length $3$. Having an orbit of length $2$ implies that $G_v$ is a $2$-group, contradicting the fact that $G_v$ has an orbit of length $3$. This concludes the proof that $\Gamma$ is $G$-arc-transitive. In particular, $G_v$ has an orbit of length $5$, and hence, by Theorem~\ref{theo:mainOrbital}, $(G,G_v)$ appears in Table~\ref{table:sporadic} or~\ref{table:infinite}. It follows that $G$ is either affine or almost simple.

If $G$ is of affine type, it has a regular elementary abelian subgroup $R$ and $\Gamma$ is a Cayley graph on $R$, with connection set $S$, say. Recall that $S$ generates $R$ and that $|S|=5$. Since $S$ is inverse-closed, this implies that $R\cong\ZZ_2^a$ for some $a\leq 5$ and thus $|G|\leq 32$. It is then easy to check that $\Gamma$  appears in Table~\ref{table:main}~(1) and, conversely, that the graph in Table~\ref{table:main}~(1) does exist and has the required properties.

We may now assume that $G$ is almost simple. If $G_v$ is not isomorphic to $\Alt(5)$ or $\Sym(5)$ then, by Tables~\ref{table:sporadic} and~\ref{table:infinite}, there are only finitely many possibilities for $\Gamma$ (in fact, it has order at most $17442$) and we can deal with them on a case-by-case basis, by computer if necessary. We obtain the graphs in Table~\ref{table:main} rows (2-4) and (6). We may therefore assume that $G_v$ is isomorphic to $\Alt(5)$ or $\Sym(5)$. In particular, $G$ appears in Table~\ref{table:maxA5} or~\ref{table:maxS5}.  Note that, in these tables, $\N_G(H)/H$ always has at most one element of order $2$. By Lemma~\ref{lemma:selfnormalising}~(\ref{two}), it follows that $|\N_G(H)/H|$ is even and that $\Gamma$ is uniquely determined by $G$ and $G_v$. Note that the number of choices for $G_v$ for a given $G$ corresponds to the number of conjugacy classes of maximal $\Alt(5)$ or $\Sym(5)$ in $G$, which is listed in the third column of Tables~\ref{table:maxA5} and~\ref{table:maxS5}, respectively. It can be checked that, in the cases where there are multiple conjugacy classes, the classes are fused by an outer automorphism of $G$ and hence the different conjugacy classes give rise to isomorphic graphs. Finally, note that the groups appearing in rows (5), (9) and (10) of Table~\ref{table:maxA5} are subgroups of the ones appearing in rows (7) and (14) of Table~\ref{table:maxS5}. In particular, the former can be ignored as $G$ will not be the (full) automorphism group of $\Gamma$ in these cases. Finally, the groups in row (4) of Table~\ref{table:maxA5} and rows (5, 7, 12, 14) of Table~\ref{table:maxS5} lead to the graphs in rows (8, 7, 9, 10, 11) of Table~\ref{table:main}. (Row (1) of Table~\ref{table:maxA5} must also be ignored for the same reason as in the last section.)

\section{Proof of Theorem~\ref{theo:max}}\label{sec:thirdmain}

Throughout this section, let $G$ be an almost simple group with socle $T$, let $M$ be a maximal subgroup of $G$ isomorphic to $ \Alt(5)$ or $\Sym(5)$, and let $H$ be a subgroup of index $5$ in $M$. We prove Theorem~\ref{theo:max} via a sequence of lemmas.

\begin{lemma}
 Theorem~\ref{theo:max} holds if $T$ is an alternating group.
\end{lemma}
\begin{proof}
Suppose that $T\cong\Alt(n)$ for some $n\geq 5$. The case $n\leq 10$ can be handled in various ways, including by computer, and we find that $G$ appears in Table~\ref{table:maxA5} rows (1,5) or Table~\ref{table:maxS5} rows (1,7). (Recall that $\Alt(6)\cong\PSL(2,9)$ and $\Sym(6)\cong\PSigmaL(2,9)$.) We thus assume that $n\geq 11$. Note that $\Alt(n)\leq G\leq\Sym(n)$ and we may view $G$ as a permutation group of degree $n$ in the natural way. If $M$ is an intransitive subgroup of $G$, then  $\Alt(k)\times \Alt(m)\leq M$ where $n=k+m$, a contradiction since $n\geq 11$. If $M$ is imprimitive, then $M=(\Sym(k)\wr \Sym(m))\cap G$ where $n=km$ and $k,m\geq 2$, so $\Sym(k)^m\cap \Alt(n)$ is a normal subgroup of $M$, a contradiction. Finally, $\Alt(5)$ and $\Sym(5)$ have no primitive actions of degree greater than $10$.
\end{proof}

\begin{lemma}
 Theorem~\ref{theo:max} holds if $T$ is a sporadic simple group.
\end{lemma}
\begin{proof}
Suppose that $T$ is a sporadic simple group. By~\cite{BrayWilson,NortonWilson}, we may assume that $T$ is not the Monster. The maximal subgroups of the remaining sporadic groups can be found in a variety of places, including \cite{onlineatlas} or the Atlas~\cite{atlas} (whose lists are not always complete).  Most of the cases can be handled in a straightforward manner using the \GAP\ package \textsc{AtlasRep}~\cite{AtlasRep}, and we find that $G$ appears in Table~\ref{table:maxA5} (2) or Table~\ref{table:maxS5} (2--5).

The only case which presents some difficulty is when $G=\Th$, the Thompson sporadic group and $M\cong\Sym(5)$. A computation yields that $|\N_G(H):H|$ is non-trivial and we give a few details. The difficulty arises because the minimum degree of a permutation representation of \thmp\ is $143\,127\, 000$. Combined with the order of $\Th$, this makes it computationally very hard to do any non-trivial calculations directly. To overcome this problem, we perform most calculations in one of the maximal subgroups of $\Th$, only ``pulling back'' to the full group when computations in two different maximal subgroups have to be reconciled.   Even using these tricks, the task is computationally non-trivial. We used \magma\ as it seems to perform better with very high degree permutation
representations than  \GAP.

It follows from the Atlas~\cite{atlas} that there is a unique choice for the conjugacy class of $M$ and, clearly, there is a unique choice for the conjugacy class of $H$ in $M$. Note that $\N_\thmp(H)$ is a $2$-local subgroup (as it normalises the Klein $4$-subgroup of $H$) and therefore, by \cite[Theorem 2.2]{wilson} it must lie in either $M_2$ or $M_3$, which are maximal subgroups of $\Th$ isomorphic to $2^5.L_5(2)$ and $2^{1+8}.A_9$ (in Atlas notation), respectively.

We then use information from the Atlas~\cite{atlas} to find a permutation representation of degree $143\,127\, 000$ for $M_2$ and $M_3$.  Despite the very high degree, the fact that the order of $M_i$ is known means that it is possible to construct a base and strong generating set for $M_i$ using randomised algorithms. It is then easy to determine the orbits of $M_i$, and by taking the action of $M_i$ on one of these orbits, obtain a faithful representation of $M_i$ of a more reasonable degree. With a representation of relatively low degree (less than $10^6$), it is possible to compute all the subgroups of $M_i$ isomorphic to $\Sym(4)$ and determine their normalisers (in $M_i$). 

Carrying out this process, we find that $M_2$ has a single conjugacy class of subgroups isomorphic to $\Sym(4)$, while $M_3$ has four such classes. To identify which of these classes contain $H$, we pull them back into the degree $143\,127\, 000$ representation of \Th. Due to the extremely high degree, it is impossible to test directly the conjugacy of these groups in \thmp, but we can compute simple invariants of them. In particular, we can determine the number of points fixed by a representative of each class. It turns out that only one conjugacy class matches the number of fixed points of $H$, thereby identifying $H$ as conjugate to a particular subgroup of $M_3$. We can then compute the normalizer in $M_3$ of $H$ to find that it has order $48$, completing the verification of Table~\ref{table:maxS5} (5).
\end{proof}

We may now assume that $T$ is a group of Lie type. By~\cite{Craven}, it is not an exceptional group, so it must be a classical group. Let $V$ be the natural module for $T$, let $n$ be the dimension of $V$, let $q$ be the order of the underlying field and $p$ its characteristic.
\begin{lemma}
If $T$ is a classical group, then either $G$ is as in Table~\ref{table:maxS5} (10,11) or $T$ is isomorphic to one of $\PSL(2,q)$ or $\PSp(6,p)$.

\end{lemma}
\begin{proof}
For a subgroup $K$ of $\PGammaL(V)$, we denote the preimage of $K$ in $\GammaL(V)$ by $\widehat{K}$. That is, $K$ is the image of $\widehat{K}$ under the homomorphism $\phi:\GammaL(V)\rightarrow \GammaL(V)/\Z(\GL(V))$. 

Suppose first that $n\leq 6$. The maximal subgroups of the classical groups of dimension at most $6$ are given in~\cite{BHRD}. The tables at the end of this book are especially useful. Care must be taken due to the fact that the tables give the structure of the pre-images in the matrix group instead of the projective group. One must also have in mind the many exceptional isomorphisms involving $\Alt(5)$ and $\Sym(5)$ (and other isomorphisms, such as $\PSp(4,2)\cong \Sym(6)$). With this in mind, one finds that, apart from the two examples which appear in Table~\ref{table:maxS5} (10,11), all examples have $T$ isomorphic to either $\PSL(2,q)$ or $\PSp(6,p)$. 

From now on, we assume that $n\geq 7$. In particular, $T\not\leq M$ and, since $M$ is maximal in $G$, $TM=G$ and $G/T\cong M/(T\cap M)$. By the Schreier Conjecture, $G/T$ is soluble, and hence $T\cap M\neq 1$.  Let $X=\soc(M)\cong \Alt(5)$. Then $X\leq T$, $M=\N_G(X)$ and $|G:T|=|M/(T\cap M)|\leqslant |M/X|\leqslant 2$. In particular, if $T=\POmega^+(8,q)$, then $G$ does not contain a triality automorphism. Our argument is aided by Aschbacher's Theorem for maximal subgroups of classical groups as developed in \cite{KL}. Since $n\geqslant 7$, either $G\leqslant \PGammaL(V)$ or $T=\PSL(V)$ and $G$ contains a graph automorphism.  In both cases, $G$ acts on the set of subspaces of $V$.

Suppose that $M$ is the stabiliser in $G$ of a nontrivial decomposition $V=U\oplus W$. Let $m=\dim(W)$. Without loss of generality, we may assume that $m\geqslant \lceil n/2\rceil \geqslant 4$. Let $\widehat{Y}$ be the subgroup of $\GammaL(W)$ induced on $W$ by $\widehat{M}_W$.  In the case where $T\neq\PSL(V)$, the maximality of $M$ implies that either $U$ and $W$ are both nondegenerate, or $U$ and $W$ are both totally singular of dimension $n/2$. If either $T=\PSL(V)$ or both $U$ and $W$ are totally singular of dimension $n/2$, then $\widehat{Y}$ contains $\SL(W)$ as a normal subgroup. However, $m\geqslant 4$, contradicting the fact that $M$ is isomorphic to one of $\Alt(5)$ or $\Sym(5)$.  Thus $T\neq \PSL(V)$ and both $U$ and $W$ are nondegenerate.  In particular, $\widehat{Y}$ contains one of $\SU(W)$, $\Sp(W)$ or $\Omega^\epsilon(W)$ as a normal subgroup. Since none of $\PSU(m,q)$  for $m\geq 4$, or $\PSp(m,q)$ for $m\geq 4$, or $\POmega^\epsilon(m,q)$ for $m\geq 5$, have $\Alt(5)$ as a composition factor, it follows that $G$ is an orthogonal group and $m=4$. Since $m\geq n/2$ it follows that $n=7$ or $8$. Thus $\widehat{M}$ contains either $\Omega(3,q)\times \Omega^{\epsilon_2}(4,q)$ or $\Omega^{\epsilon_1}(4,q)\times \Omega^{\epsilon_2}(4,q)$ as a normal subgroup.  Note that, if $n=7$, then $q$ is odd. Also $\Omega(3,q)\cong \PSL(2,q)$ for $q$ odd, $\Omega^-(4,q)\cong \PSL(2,q^2)$ and $\Omega^+(4,q)\cong \SL(2,q)\circ \SL(2,q)$. Since $M$ is insoluble and has $\Alt(5)$ as a unique non-abelian composition factor it follows that $n\neq 7$. Moreover, when $n=8$ we must have that $\epsilon_1=+$, $\epsilon_2=-$ and $q=2$. In this case, the stabiliser of a decomposition in $G$ will be $3$-local (as $\PSL(2,2)\cong \Sym(3)$) which $M$ is not. This contradiction completes the proof that $M$ is not the stabiliser in $G$  of a decomposition $V=U\oplus W$.

Suppose now that $M$ fixes some nontrivial subspace $U$.  As $M$ is maximal in $G$, it is the stabiliser of $U$ in $G$. Since $M$ is not $p$-local, $M$ is not a parabolic subgroup. In particular,  $T\neq \PSL(V)$ and $U$ is either nondegenerate or $p=2$, $G$ is an orthogonal group and $U$ is a nonsingular $1$-space. The latter is not possible as the stabiliser of such a $1$-space in $\POmega^{\pm}(n,q)$  is isomorphic to $\Sp_{n-2}(q)$, which is not contained in $\Sym(5)$. It follows that $U$ is nondegenerate and  $M$ also fixes $U^\perp$ and the decomposition $V=U\oplus U^\perp$. This contradicts the previous paragraph.

We may now assume that $M$ does not fix any nontrivial subspace of $V$. Suppose that, on the other hand,  $X$ does fix  a nontrivial subspace $U$. Since $M=\N_G(X)$,  there is another subspace $W$ fixed by $X$ such that $M$ fixes the set $\{U,W\}$. Moreover, as $M$ is maximal in $G$, it is  the stabiliser in $G$ of $\{U,W\}$ and either $U<W$ or $V=U\oplus W$. The latter case contradicts an earlier statement.  In the former case, since $M$ does not fix $W$, we must have that $T=\PSL(V)$ and $G$ contains a graph automorphism (recall that $n\geq7$ so $T\not\cong\PSp(4,q)$). However, this contradicts $M$ not being $p$-local.

We have shown that $X$ does not fix any nontrivial subspace of $V$ and hence $\widehat{X}$ is irreducible.  By \cite[31.1]{Asch}, we have $\widehat{X}=\widehat{X}'\circ \Z(\widehat{X})$. Since  $\Z(\widehat{X})$ consists of scalars, it follows that $\widehat{X}'$ is irreducible on $V$.  Moreover, since $\widehat{X}'$ is a perfect central extension of $\Alt(5)$, it is isomorphic to $\Alt(5)$ or $2\nonsplit \Alt(5)$. By the Brauer character table of $\widehat{X}'$,  the  (not necessarily absolutely) irreducible representations of $\widehat{X}'$ have dimension at most $6$, contradicting our assumption that $n\geq 7$.
\end{proof}
The next lemma follows from Dickson's classification of the subgroups of $\PGL(2,q)$ \cite{dickson}.
\begin{lemma}\label{lem:normA4}
The subgroup of $\PSL(2,q)$ isomorphic to $\Alt(4)$ is self-normalising if and only if $q$ is even or $q\equiv \pm 3 \pmod 8$. For $q$ odd, $\Sym(4)$ is a self-normalising subgroup of $\PGL(2,q)$ and it is the normaliser of an $\Alt(4)$. 
\end{lemma}

\begin{lemma}
 Theorem~\ref{theo:max} holds when $T\cong\PSL(2,q)$.
\end{lemma}
\begin{proof}
Since $\PSL(2,5)\cong \Alt(5)$, we see from \cite[Table 8.1]{BHRD} that $\PSL(2,5^2)$ has two classes of maximal subgroups isomorphic to $\Sym(5)$. This gives row (6) of Table~\ref{table:maxS5}. The same isomorphism also yields that there is a unique conjugacy class of maximal $\Alt(5)$ subgroups in $\PSL(2,5^r)$ for $r$ an odd prime and a unique conjugacy class of maximal $\Sym(5)$ subgroups in $\PGL(2,5^r)$ (and no such maximal subgroups when $q=p^r$ with $r$ not prime). Since $r$ is odd, $5^r\equiv -3\pmod 8$ and Lemma~\ref{lem:normA4} implies  row (7) of Table~\ref{table:maxA5} and row (9) of Table~\ref{table:maxS5}. 

Since $\PSL(2,4)\cong \Alt(5)$, we see from \cite[Table 8.1]{BHRD} that $\Alt(5)$ is a maximal subgroup of $\PSL(2,2^{2r})$ for $r$ an odd prime and there is a unique conjugacy class of such subgroups. Such a subgroup is normalised by a field automorphism of $T$ of order $2r$. When $r=2$, such an $\Alt(5)$ is the centraliser of the field automorphism of order two but when $r$ is odd the centraliser of a field automorphism of order two is $\PSL(2,2^r)$, which does not contain an $\Alt(5)$. Thus when $r$ is odd, the normaliser of $\Alt(5)$ in $\PSL(2,2^{2r}).2$ is $\Sym(5)$ and is a maximal subgroup. Again there is a unique conjugacy class of such subgroups. Lemma~\ref{lem:normA4} then yields  row (6) of Table~\ref{table:maxA5} and row (8) of Table~\ref{table:maxS5}.

Using \cite[Table 8.1]{BHRD}, we see that $\Alt(5)$ is a maximal subgroup of $\PSL(2,p)$ for $p\equiv \pm 1\pmod{10}$. There are two classes of such maximals and they are self-normalising in $\PSL(2,p)$. This gives rows (3) and (4) of Table~\ref{table:maxA5} with the normaliser of an $\Alt(4)$ given by Lemma~\ref{lem:normA4}. We also see that there are two classes of maximal $\Alt(5)$ subgroups in $\PSL(2,p^2)$ when $p\equiv \pm 3\pmod{10}$. Since $p^2\equiv 1\pmod 8$, by Lemma~\ref{lem:normA4} the normaliser in $T$ of an $\Alt(4)$ is $\Sym(4)$ and we get row (5) of Table~\ref{table:maxA5}. Finally, each of these $\Alt(5)$ subgroups is normalised but not centralised by a field automorphism. Hence we obtain two conjugacy classes of maximal $\Sym(5)$ subgroups in $\PSigmaL(2,p^2)$. The normaliser of an $\Alt(4)$ in $\PSigmaL(2,p^2)$ is then $\Sym(4)\times \ZZ_2$ and hence the $\Sym(4)$ in $\PSigmaL(2,p^2)$ has normaliser twice as large. Hence we have row (7) of Table~\ref{table:maxS5}.
\end{proof}

Before dealing with the case where $T\cong\PSp(6,p)$ we need a couple of lemmas. For a group $X$ fixing a set $U$, we denote the permutation group of $X$ induced on $U$ by $X^U$.

\begin{lemma}\label{semisimple} \cite[p.36]{Wall}
Let $p$ be an odd prime. A semisimple element $A$ of $\GL(d,p)$ is conjugate to an element of $\Sp(d,p)$ if and only if $A$ is conjugate to $(A^{-1})^T$.
\end{lemma}

\begin{lemma}\label{lem:spsubs}\cite[Lemmas 4.1.1 and 4.1.12]{KL}
 Let $X\leqslant \Sp(d,p)$ and suppose that $X$ fixes an $m$-dimensional subspace $U$ of the natural module. If $X^U$ is irreducible, then $U$ is either nondegenerate or totally isotropic. Moreover
\begin{enumerate}[{\rm (i)}]
 \item if $U$ is nondegenerate then $\Sp(d,p)_U\cong \Sp(m,p)\times \Sp(d-m,p)$;
 \item if $(|X|,p)=1$ and $U$ is totally isotropic, then $X$ fixes another totally isotropic subspace $U^*$ such that $U$ and  $U^*$ are disjoint and $\dim U = \dim U^*$. Moreover, 
$$(\Sp(d,p)_{U\oplus U^*})^{U\oplus U^*}=\left\{ \begin{bmatrix} A &0 \\0 &(A^{-1})^T\end{bmatrix} \mid A\in\GL(m,p)\right\}.$$
\end{enumerate}
\end{lemma}

\begin{lemma}
 Theorem~\ref{theo:max} holds when $T\cong\PSp(6,p)$.
\end{lemma}

\begin{proof}
The cases when $p\leq 5$ can be verified by a \magma\ calculation. (In this case, we obtain row (8) of Table~\ref{table:maxA5} and row (13) of Table~\ref{table:maxS5}.) We assume now that $p\geq 7$.

Suppose first that $p\equiv \pm 1\pmod{8}$. By \cite[Table 8.29]{BHRD},  $M\cong\Sym(5)$, $G=T$ and there are two possibilities for the conjugacy class containing $M$. Let $\widehat{M}\cong 2\nonsplit \Sym(5)^{-}$  be the preimage of $M$ in $\Sp(6,p)$ and let $\widehat{H}\cong 2\nonsplit \Sym(4)^{-}$ be the index five subgroup of $\widehat{M}$ corresponding to $H$. Note that $V\downarrow \widehat{M}$ is absolutely irreducible. 
By considering the Brauer character tables for $2\nonsplit \Sym(5)^{-}$ and $2\nonsplit \Sym(4)^{-}$,  we deduce that  
$V\downarrow\widehat{H}=U\oplus W$ where $U$ and $W$ are absolutely irreducible representations of $\widehat{H}$ over $\GF(p)$ with degree two and four respectively.
Since $p\geq 7$, Lemma~\ref{lem:spsubs} implies that $U$ and $W$ are nondegenerate and hence the stabiliser in $\Sp(6,p)$ of this decomposition is $\Sp(2,p)\times \Sp(4,p)$. Since $\widehat{H}$ is absolutely irreducible on $U$ and $W$, it follows from Schur's Lemma that the centraliser of $\widehat{H}$ in $\Sp(6,p)$ is $\Z(\Sp(2,p))\times \Z(\Sp(4,p))$. By Lemma~\ref{lem:normA4}, $\Sym(4)$ is self-normalising in $\PSL(2,p)$ and hence $\widehat{H}$ is self-normalising in $\Sp(2,p)$. Thus $\N_{\Sp(6,p)}(\widehat{H})=\C_{\Sp(6,p)}(\widehat{H})\widehat{H}$ and so $|\N_{G}(H):H|=2$. This verifies Row (12) of Table~\ref{table:maxS5}.

Next suppose that $p\equiv \pm 3\pmod{8}$. It follows from \cite[Table 8.29]{BHRD} that $M\cong\Alt(5)\leqslant\PSp(6,p)$. Let $\widehat{M}\cong 2 \nonsplit \Alt(5)$ be the preimage of $M$ in $\Sp(6,p)$.  When $p\equiv \pm3,\pm13\pmod{40}$, \cite[Table 8.29]{BHRD} asserts that $M$ is maximal in $\PSp(6,p)$ and, moreover, $X:=\N_{\PGSp(6,p)}(M)\cong \Sym(5)$ is a maximal subgroup of $\PGSp(6,p)$.
  When $p\equiv \pm 11,\pm 19\pmod {40}$,  $M$ is not maximal in $\PSp(6,p)$ but $X:=\N_{\PGSp(6,p)}(M)\cong \Sym(5)$ is maximal in $\PGSp(6,p)$.  As usual, we denote the preimage of $X$ in $\GSp(6,p)$ by $\widehat{X}$. Let $\widehat{H}=2 \nonsplit \Alt(4)$ be the subgroup of $\widehat{M}$ corresponding to $H$. Note that $V\downarrow \widehat{M}$ is absolutely irreducible. Let $\chi$ be the character for $V\downarrow \widehat{M}$ and let $F$ be a splitting field for $2 \nonsplit \Alt(4)$. By the Brauer character table of $2\nonsplit \Alt(4)$,  we conclude that $\chi=\chi_1+\chi_2+\chi_3$ over $F$, where the $\chi_i$ are the  three irreducible representations of $2 \nonsplit \Alt(4)$ of degree two. Moreover, when $p\equiv 1 \pmod 3$, we may take $F=\GF(p)$, and when $p\equiv 2 \pmod 3$, we may take $F=\GF(p^2)$. 
We divide our analysis into these two cases.

Suppose first that $p\equiv 1 \pmod 3$ and $F=\GF(p)$. In this case $V$ splits as the sum of three irreducible spaces $U, W_1$ and $W_2$ for $\widehat{H}$ of dimension two. By looking at the character tables and using Lemmas~\ref{semisimple} and~\ref{lem:spsubs}, it follows that $U$ is nondegenerate while $W_1$ and $W_2$ are complementary totally isotropic subspaces. By Lemma~\ref{lem:spsubs}, the partwise stabiliser in $\Sp(6,p)$ of the decomposition of $V$ preserved by $\widehat{H}$ is $\Sp(2,p)\times \GL(2,p)$.  Since the actions of $\widehat{H}$ on $W_1$ and $W_2$ are dual, the centraliser in $\Sp(6,p)$ of $\widehat{H}=2\nonsplit \Alt(4)$ is $Z_1\times \Z(\GL(2,p))$ where $Z_1=\Z(\Sp(2,p))$.  By Lemma~\ref{lem:normA4}, $\Alt(4)$ is self-normalising in $\PSp(2,p)\cong\PSL(2,p)$ when $p\equiv \pm 3\pmod{8}$ and hence $\N_{\Sp(6,p)}(\widehat{H})=\widehat{H}\C_{\Sp(6,p)}(\widehat{H})$. Thus $\N_{\PSp(6,p)}(H)/H$ is a cyclic group of order $p-1$. This verifies row (9) of Table~\ref{table:maxA5}.

Now consider $\widehat{X}$. It has an index five subgroup $\widehat{R}$ containing $Z=\Z(\GSp(6,p))\cong \ZZ_{p-1}$ such that $R=\widehat{R}/Z\cong \Sym(4)$ and $\widehat{R}$ normalises $\widehat{H}$. Now $\widehat{R}$ must preserve the decomposition $V=U\perp (W_1\oplus W_2)$ fixed by $\widehat{H}$. The partwise stabiliser of this partition in $\GSp(6,p)$ is $(\Sp(2,p)\times \GL(2,p))\rtimes \langle \delta\rangle$ where $\delta$ is an element of order $p-1$ that centralises the $\GL(2,p)$ and generates $\GSp(2,p)$ with $\Sp(2,p)$. Since $\widehat{R}$ contains an element that does not centralise $\widehat{H}$, it follows that $\widehat{R}$ must interchange $W_1$ and $W_2$. In particular, $|\C_{Z}(\widehat{R})|=2$. It follows that $\C_{\GSp(6,p)}(\widehat{R})=Z_1Z$ and hence $|\N_{\PGSp(6,p)}(R):R|=2$.  This verifies row (14) of Table~\ref{table:maxS5} when $p\equiv 1 \pmod 3$.

We now assume that $p\equiv 2 \pmod 3$ and  $F=\GF(p^2)$. It follows from the Brauer character table of $2\nonsplit \Alt(4)$ that  $\chi_1$ can be realised over $\GF(p)$ while $\chi_2$ and $\chi_3$ cannot, hence  the restriction of $V$ to $\widehat{H}$ must decompose as $V=U\oplus W$ with $\dim(U)=2$ and $\dim(W)=4$. Since $\dim(U)\neq \dim(W)$ it follows from Lemma~\ref{lem:spsubs} that $U$ and $W$ are both nondegenerate and hence the stabiliser of this decomposition in $\Sp(6,p)$ is $\Sp(2,p)\times\Sp(4,p)$. Moreover, the image of $2\nonsplit \Alt(4)$ in the group induced on $W$ is contained in the subgroup $\ZZ_{p+1}\circ \Sp(2,p^2)$. Thus the centraliser of $\widehat{H}$ in $\Sp(6,p)$ is equal to $Z_1\times Z_2$ where $Z_1=\Z(\Sp(2,q))$ and $Z_2$ has order $p+1$. Since $p\equiv \pm 3\pmod{8}$, we again have that $\N_{\Sp(6,p)}(\widehat{H})=\widehat{H}\C_{\Sp(6,p)}(\widehat{H})$ and hence $\N_{\PSp(6,p)}(H)/H$ is cyclic of order $p+1$.  This verifies row (10) of Table~\ref{table:maxA5}.

Now consider $\widehat{X}$ and again let $Z=\Z(\GSp(6,p))$. Again it has an index five subgroup $\widehat{R}$ containing $Z$ such that $R=\widehat{R}/Z\cong \Sym(4)$ and $\widehat{R}$ normalises $\widehat{H}$. Also $\widehat{R}$ must preserve the decomposition of $V=U\oplus W$ preserved by $\widehat{H}$. The stabiliser in $\GSp(6,p)$ of this decomposition is $(\Sp(2,p)\times\Sp(4,p))\rtimes\langle \delta\rangle$ where $\delta$ has order $p-1$ and acts as an outer automorphism of order $p-1$ on both $\Sp(2,p)$ and $\Sp(4,p)$.  Consider $\widehat{H}$ acting on $V'=V\otimes \GF(p^2)$ as a 6-dimensional space over $\GF(p^2)$. Since $\GF(p^2)$ is a splitting field for $\widehat{H}$, we have that $\widehat{H}$ decomposes $V'$ as a nondegenerate $2$-space and two totally isotropic 2-spaces. The partwise stabiliser in $\GSp(6,p^2)$ of this decomposition is $(\Sp(2,p^2)\times \GL(2,p^2))\rtimes\langle \delta\rangle$ where $\delta$ is an element of order $p^2-1$ that centralises the $\GL(2,p^2)$ and, together with $\Sp(2,p^2)$, generates $\GSp(2,p^2)$. Since $\widehat{R}\backslash \widehat{H}$ contains an element that does not centralise $\widehat{H}$, it follows that $\widehat{R}$ (when  viewed as acting on $V'$)  must interchange the two totally isotropic $2$-spaces. Thus $\widehat{R}$ is absolutely irreducible on $W$. Hence  $\C_{GSp(6,p)}(\widehat{R})=Z_1Z$ and $|\N_{\PGSp(6,p)}(R):R|=2$. This completes the verification of row (14) of Table~\ref{table:maxS5}.
\end{proof}

\section{Proof of Corollary~\ref{cor:halfarc}}\label{sec:HAT}

Suppose, to the contrary, that $\Gamma$ is a half-arc-transitive vertex-primitive graph of valency $10$, let $G$ be its automorphism group, and let $(u,v)$ be an arc of $\Gamma$. Let $\vGa$ be the digraph with the same vertex-set $\V\Gamma$ as $\Gamma$ and with arc-set $(u,v)^G$. Note that $\vGa$ is an asymmetric $G$-arc-transitive digraph of out-valency $5$. In particular, $G_v$ has an orbit of length $5$ and  $(G,G_v)$ appears in Table~\ref{table:sporadic} or~\ref{table:infinite}. It follows that $G$ is either affine or almost simple.

If $G$ is of affine type, it has a regular elementary abelian subgroup $R$ and $\Gamma$ is a Cayley graph on $R$, with connection set $S$, say. Since $R$ is abelian, the permutation sending every element of $R$ to its inverse  is an automorphism of $\Gamma$. On the other hand, if $s\in S$, then the composition of the inversion map with multiplication by $s$ is an automorphism of $\Gamma$ that reverses the arc $(1,s)$, contradicting the fact that $\Gamma$ is half-arc-transitive.

We may now assume that $G$ is almost simple. If $G_v$ is not isomorphic to $\Alt(5)$ or $\Sym(5)$ then, as in the proof of Theorem~\ref{theo:mainGraph}, there are only finitely many possibilities which can be handled on a case-by-case basis. These yield no examples. We may therefore assume that $G_v$ is isomorphic to $\Alt(5)$ or $\Sym(5)$. In particular, by Theorem~\ref{theo:max}, $G$ appears in Table~\ref{table:maxA5} or~\ref{table:maxS5}. By Lemma~\ref{lemma:selfnormalising}~(\ref{one}-\ref{two}), we may restrict our attention to rows where $\N_G(H)/H$ contains an element of order at least $3$. In particular, $G\cong\PSp(6,p)$ for some prime $p$ with $p\equiv \pm 3,\pm 13 \pmod{40}$ and $G_v\cong\Alt(5)$. 

Let $H=G_{uv}$ and note that $H\cong\Alt(4)$. Let $G^*=\PGSp(6,p)$.  By \cite[Table 8.29]{BHRD},
 $\N_{G^*}(G_v)\cong \Sym(5)$ and hence $G^*\leq\N_{\Sym(\V\Gamma)}(G)$ and $G^*_v=\N_{G^*}(G_v)$.  Let $\Delta$ be the orbit of $G^*_v$ containing $u$. If $\Delta$ is also an orbit of $G_v$, then $(u,v)^{G^*}=(u,v)^G$ and $G^*$ is contained in the automorphism group $G$ of $\Gamma$, a contradiction. Since $|G^*_v:G_v|=2$, the only other possibility is that $\Delta$ is a union of two orbits of $G_v$ of the same size, namely $5$. In particular $|\Delta|=10$.  It follows that $G^*_{uv}=H$.  Let $H^*=\N_{G^*_v}(H)$. Note that $H^*\cong\Sym(4)$. Since $H$ is a characteristic subgroup of $H^*$, we have that $\N_{G^*}(H)=\N_{G^*}(H^*)$. If $p=3$, then Table~\ref{table:maxS5} implies that $\N_{G^*}(H)=H^*$. If $p\neq 3$, then it follows by Table~\ref{table:maxS5} (and the fact that $\Sym(4)$ is a complete group) that $\N_{G^*}(H)=H^*\times Z$ for some $Z\cong\ZZ_2$. In both cases, we have that $\N_{G^*}(H)/H$ is an elementary abelian $2$-group.

Let $\Gamma^*$ be the digraph with vertex-set $\V\Gamma$ and with arc-set $(u,v)^{G^*}$. Since $|\Delta|=10$, $\Gamma^*$ has out-valency $10$. Let $w'$ be an out-neighbour of $v$. As $\Gamma^*$ is $G^*$-arc-transitive, $H$ and $G^*_{vw'}$ are conjugate in $G$ and, in particular, isomorphic. On the other hand, $G^*_v$ has a unique conjugacy class of subgroups isomorphic to $\Alt(4)$, and hence $H$ and $G^*_{vw'}$ are conjugate in $G^*_v$. It follows that $H=G^*_{vw}$ for some out-neighbour $w$ of $v$ in $\Gamma^*$. Since $\Gamma^*$ is $G^*$-arc-transitive, there exists $g\in G^*$ such that $(u,v)^g=(v,w)$. Note that $g$ normalises $H$. By the previous paragraph, this implies that $g^2\in H$. However $u^{g^2}=v^g=w$ and so $u=w$ and $\Gamma^*$ is actually a graph. Since $G<G^*$, $\vGa$ is a sub-digraph of $\Gamma^*$ and hence $\Gamma^*=\Gamma$. This implies that $G^*$ is contained in the automorphism group of $\Gamma$ which is a contradiction.

\section{Proof of Corollary~\ref{cor:halfarc2}}\label{sec:HAT2}

We first need the following lemma.
\begin{lemma}\label{subdeg6}
Let $p$ be a prime with $p\equiv 7,23\pmod{40}$, let $G=\PSp(6,p)$ and let $M$ be a  maximal subgroup of $G$ isomorphic to $\Sym(5)$. If $H$ is a subgroup of index $6$ in $M$, then $\N_G(H)/H\cong \ZZ_{p+1}$.
\end{lemma}
\begin{proof}
First, note that $M$ actually exists by Theorem~\ref{theo:max}. Note also that $\Sym(5)$ has a unique conjugacy class of subgroups of index $6$. These subgroups are maximal and isomorphic to $\AGL(1,5)$. Let $\widehat{M}$ and $\widehat{H}$ be the preimage of $M$ and $H$ in $\Sp(6,p)$, respectively. Note that $\widehat{M}\cong 2\nonsplit \Sym(5)^-$ and $\widehat{H}\cong \ZZ_5\rtimes\ZZ_8$. Let $V$ be the natural $6$-dimensional module for $\Sp(6,p)$ over $\GF(p)$. Since $p\equiv 7\pmod{8}$,  it follows from the Brauer character tables of $\widehat{M}$ and $\widehat{H}$ (available in \textsc{Magma}, for example) that  $V\downarrow \widehat{H}$ splits as  a sum of an absolutely irreducible $4$-dimensional subspace $W$ and an irreducible but not absolutely irreducible subspace $U$ of dimension $2$.
Moreover, $\widehat{H}$ is faithful on $W$ while elements of order 5 in $\widehat{H}$ act trivially on $U$. Since $p\geq 7$, Lemma~\ref{lem:spsubs} implies that $U$ and $W$ are nondegenerate and hence the stabiliser in $\Sp(6,p)$ of this decomposition is $\Sp(2,p)\times \Sp(4,p)$. 
By Schur's Lemma, $\C_{\Sp(6,p)}(\widehat{H})=Z_1\times Z_2$ where $Z_1\cong\ZZ_{p+1}$ and $Z_2=\Z(\Sp(4,p))\cong\ZZ_2$. Since elements of order $5$ in $\widehat{H}$ act trivially on $U$, any element of $\N_{\Sp(6,p)}(\widehat{H}) \setminus \C_{\Sp(6,p)}(\widehat{H})\widehat{H}$, must centralise the elements of order $5$ in $\widehat{H}$.  Moreover, for $p\equiv -1\pmod 8$, the normaliser in $\Sp(2,p)$ of a cyclic group of order 8 is $Q_{2(p+1)}$. Now $5$ divides $p^2+1$, and so the centraliser $C$ of an element of order $5$ in $\Sp(4,p)$ is cyclic of order $p^2+1$ (see \cite[Proposition 3.4.3 and Remark 3.4.4]{BG}). However, 4 does not divide $p^2+1$ and so the Sylow 2-subgroup of $C$  is equal to $\Z(\Sp(4,p))=Z_2$. Thus $\N_{\Sp(6,p)}(\widehat{H})=C_{\Sp(6,p)}(\widehat{H})\widehat{H}$ and so 
$$\N_G(H)/H\cong\N_{\Sp(6,p)}(\widehat{H})/\widehat{H}\cong C_{\Sp(6,p)}(\widehat{H})/\Z(\widehat{H})=(Z_1\times Z_2)/Z_2\cong\ZZ_{p+1}.$$
\end{proof}

We are now ready to prove Corollary~\ref{cor:halfarc2}. Let $p$ be a prime with $p\equiv 7,23\pmod{40}$ and let $G=\PSp(6,p)$. By Theorem~\ref{theo:max}, there exists  a maximal subgroup $M$ of $G$ isomorphic to $\Sym(5)$. Note that $\Sym(5)$ has a unique conjugacy class of subgroups of index $6$, and these subgroups are maximal and not normal. Let $H$ be a subgroup of index $6$ in $M$. By Lemma~\ref{subdeg6}, $\N_G(H)/H\cong \ZZ_{p+1}$. By Lemma~\ref{lemma:selfnormalising}, there exists $\Gamma'$ a $G$-arc-transitive digraph of out-valency $6$ that is not a graph. Let $\Gamma$ be the underlying graph of $\Gamma'$ and let $A$ be the automorphism group of $\Gamma$. Note that $\Gamma$ has valency $12$ and  that $G\leqslant A\leqslant \Sym(V\Gamma)$. Since $\Alt(V\Gamma)\not\leqslant A$, it follows from \cite{LPS}  that $\soc(A)=G$ and thus $A=G$ or $A=\PGSp(6,p)$. However, by \cite[Table 8.29]{BHRD}, $M$ is self-normalising in $\PGSp(6,p)$ and thus $\PGSp(6,p)\not\leqslant\Sym(V\Gamma)$. It follows that $A=G$ and hence $\Gamma$ is half-arc-transitive. 

As there are infinitely many primes $p$ with $p\equiv 7,23\pmod{40}$, this proves that there are infinitely many vertex-primitive half-arc-transitive graphs of valency $12$, as required.

\medskip

\noindent\textsc{Acknowledgements.}
This research was supported by the Australian Research Council grants DE130101001, DP130100106 and DP150101066. We would like to thank the Centre for the Mathematics of Symmetry and Computation for its support, as this project started as a problem for the annual CMSC retreat in 2015.  We would also like to thank David Craven for his immense and timely help in dealing with the case of exceptional groups of Lie type, and Derek Holt for assistance with the computations relating to the Thompson group.

\end{document}